\documentclass{article}
\usepackage[T2A]{fontenc}
\usepackage[utf8]{inputenc}
\usepackage{amsfonts}
\usepackage{amsthm}
\usepackage{amsmath}
\usepackage{amssymb}
\usepackage{hyperref}
\usepackage{pdfpages}
\usepackage{dsfont}
\usepackage{scalerel}
\usepackage{enumitem}   
\usepackage{makecell,xcolor}
\usepackage{tikz}
\usetikzlibrary{matrix,decorations.pathreplacing, calc, positioning, fit, backgrounds}

\usepackage{amsthm,amsmath,amsfonts,caption}
\usepackage[top=1.25in,bottom=1in,left=1in,right=1in]{geometry}

\newtheorem{theorem}{Theorem}
\newtheorem*{theorem*}{Theorem}
\newtheorem{lemma}{Lemma}

\newtheorem*{remark}{Remark}

\begin{document}

\title{On the number of spanning trees in bipartite graphs}
\author{Albina Volkova}
\date{}

\maketitle

\begin{abstract}
In this paper, we address the Ehrenborg’s conjecture which proposes that for any bipartite graph $G$ the inequality $\tau(G) \leq m^{-1}n^{-1}D(G)$ holds. Here  $\tau (G) $ denotes the number of spanning trees of $G$, $D(G)$ is the product of the degrees of the vertices and $m, n$ are the sizes of the components.
We show that the conjecture is true for a one-side regular graph (that is a graph for which all degrees of the vertices of at least one of the components are equal). We also present a proof alternative to one given in \cite{Ferrers} of the fact that the equality holds for Ferrers graphs.
\end{abstract}

\section{Introduction and preliminaries}

\par For a connected undirected graph $ G = (V, E) $, we denote by $ \tau (G) $ the number of spanning trees of $G$. In this paper, we address the problem of bounding $ \tau (G) $ from above. 

\par There are many different upper bounds for the number of spanning trees in terms of degrees of vertices, number of edges, number of vertices, and other graph characteristics.
A common approach for proving upper bounds on $ \tau (G) $ is based on a statement known as \textit{Kirchhoff's Matrix-Tree Theorem} \cite{Kirh}. Let us recall its formulation.

The \textit{Laplacian} matrix of the graph
$G=(V,E)$ is the operator acting on the space
$\mathbb{R}^V$ of functions acting on $V$ that maps a function $f(v)$ to
$$
(\Delta f)(v):=\deg(v)f(v)-\sum_{u:uv\in E} f(u).
$$ 
If the vertices are enumerated as $V=\{v_1,\ldots,v_n\}$ then the Laplacian matrix in a natural basis has the form
\[
    a_{ij}= 
\begin{cases}
    \deg(v_i), & \text{if } i=j\\
     -1 ,& \text{if } v_i v_j \in E\\
     0 & \text{otherwise.}
\end{cases}
\]

The Matrix-Tree Theorem states that $ \tau(G) $ can be computed as the cofactor of any entry in the Laplacian matrix of $ G $.

\par It is easy to see that the Laplacian is always a singular matrix. It is also known (see the paper \cite{CDS80}) that $ \tau(G) $ can be expressed in terms of the eigenvalues of this matrix as follows
$$
\tau(G) = \frac{1}{n} \prod_{i=1}^{n-1} \mu_i,
$$
where $\mu_1 \geq \dots \geq \mu_{n-1} > \mu_n = 0$
are the eigenvalues of the Laplacian of $G$.
\par Hence our problem can be viewed as a
linear-algebraic problem of estimating the product of nonzero eigenvalues of the Hermitian singular matrix. One of the first upper bounds \cite{GM} which was obtained using this approach is as follows:
 \begin{align}\label{eq:GM}
\tau(G) \leq \left(\frac{n}{n-1}\right)^{n-1} \left(\frac{\prod_{i=1}^n d_{i}}{2|E|}\right).
\end{align}
\par The key idea behind the bound above is to apply the AM-GM inequality to the normalized eigenvalues. Other recent advances in constructing upper bounds for the number of spanning trees can be found in \cite{KS19}, another overview is given in \cite{Bozkurt}.
\par In this paper, we focus on giving an upper bound for $ \tau(G) $ in case of bipartite graph. In what follows, we will consider a connected finite bipartite graph $ G = (V, E) $,  
$n + m =|V|$, where $m$ and $n$ are the sizes of the components. The degrees of the vertices in the first and the second components are denoted as
$\{ a_1, a_2, \dots, a_n\}$ and $\{b_1, b_2, \dots, b_m\}$, respectively. In addition, we introduce the notation for the product of the degrees of the vertices $D(G) = a_1a_2\dots a_n \cdot b_1b_2\dots b_m$.
\par In the above notation Richard Ehrenborg formulated a conjecture according to which for any bipartite graph the following inequality holds
\begin{align}\label{eq:Econj}
    \tau(G) \leq m^{-1}n^{-1}D(G).
\end{align}

\par In \cite{Ferrers} Ehrenborg and Willigenburg proved that the inequality \eqref{eq:Econj} is sharp for Ferrers graphs. According to \cite{Ferrers} we define a \textit{Ferrers graph} to be a bipartite graph in which the vertices of both components can be ordered in such a way that if there is an edge $(a,b)$, then all edges $(a',b)$ and $(a,b')$ for $a'<a$ and $b'<b$ are also present in the graph (see Figure \ref{fig:Ferrers_ex}).

\begin{figure}[h]
\begin{center}
\begin{tabular}{>{\centering\arraybackslash}m{.48\textwidth}>{\centering\arraybackslash}m{.48\textwidth}}
\begin{tikzpicture}[scale=.6]
\draw (0,0) -- (4,0);
\draw (0,-1) -- (4,-1);
\draw (0,-2) -- (4,-2);
\draw (0,-3) -- (3,-3);
\draw (0,-4) -- (3,-4);
\draw (0,-5) -- (1,-5);

\draw (0,0) -- (0,-5);
\draw (1,0) -- (1,-5);
\draw (2,0) -- (2,-4);
\draw (3,0) -- (3,-4);
\draw (4,0) -- (4,-2);

\draw (-.5,-.5) node {$a_1$};
\draw (-.5,-1.5) node {$a_2$};
\draw (-.5,-2.5) node {$a_3$};
\draw (-.5,-3.5) node {$a_4$};
\draw (-.5,-4.5) node {$a_5$};

\draw (0.5,0.5) node {$b_1$};
\draw (1.5,0.5) node {$b_2$};
\draw (2.5,0.5) node {$b_3$};
\draw (3.5,0.5) node {$b_4$};
\end{tikzpicture}
&
\begin{tikzpicture}[scale=1.0]
\foreach \t in {(0,0), (1,0), (2,0), (3,0), (4,0), (.5,1.5), (1.5,1.5), (2.5,1.5), (3.5,1.5)}{
	\draw[fill=black] \t circle (.1);
}
\draw (0,0) node[anchor = north] {$a_1$};
\draw (1,0) node[anchor = north] {$a_2$}; 
\draw (2,0)  node[anchor = north] {$a_3$};
\draw (3,0)  node[anchor = north] {$a_4$};
\draw (4,0)  node[anchor = north] {$a_5$};
\draw (.5,1.5)  node[anchor = south] {$b_1$};
\draw (1.5,1.5) node[anchor = south] {$b_2$}; 
\draw (2.5,1.5) node[anchor = south] {$b_3$};
\draw (3.5,1.5) node[anchor = south] {$b_4$};

\draw (0,0) -- (.5,1.5);
\draw (0,0) -- (1.5,1.5);
\draw (0,0) -- (2.5,1.5);
\draw (0,0) -- (3.5,1.5);

\draw (1,0) -- (.5,1.5);
\draw (1,0) -- (1.5,1.5);
\draw (1,0) -- (2.5,1.5);
\draw (1,0) -- (3.5,1.5);

\draw (2,0) -- (.5,1.5);
\draw (2,0) -- (1.5,1.5);
\draw (2,0) -- (2.5,1.5);

\draw (3,0) -- (.5,1.5);
\draw (3,0) -- (1.5,1.5);
\draw (3,0) -- (2.5,1.5);

\draw (4,0) -- (.5,1.5);
\end{tikzpicture}
\end{tabular}
\end{center}
\caption{The Ferrers graph and the Ferrers diagram associated with it.}
\label{fig:Ferrers_ex}
\end{figure}
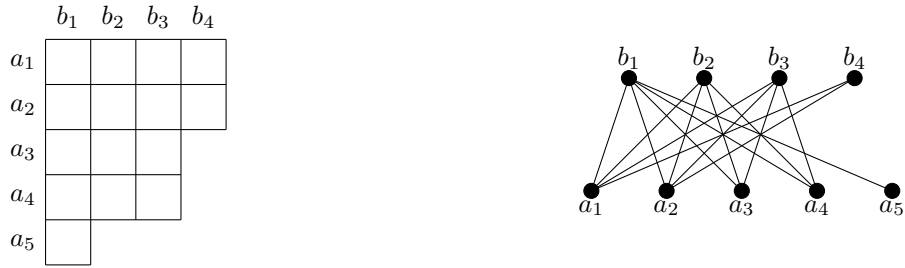

\par In Section \ref{sec:sharp}, we give an alternative proof of the sharpness of the bound \eqref{eq:Econj}. 
\par Unfortunately, Ehrenborg's conjecture remains open at the moment. Though, for a bipartite graph we know some improvement of the bound \eqref{eq:GM} . For example, in \cite{Bozkurt} the following upper bound was shown for a bipartite connected graph 
\begin{equation}\label{eq:Bozkurt}
    \tau(G) \leq \frac{\prod_{i=1}^n a_{i} \prod_{i=1}^m b_{i}}{|E|}.
\end{equation}

\par Section \ref{sec:BoundsRegular} of this paper
is devoted to Ehrenborg's conjecture in the case of a one-side regular graph (that is a graph for which all degrees of the vertices of at least one of the components are equal). We also give some generalizations of this result, which allow us to bound the number of spanning trees with certain constraints on the degree of a particular vertex.

Let's agree on some additional
notation. We use $[n]$ as a standard notation for $\{1,\ldots,n\}$, where $ n $ is a natural number;
$e_1,\ldots,e_n$ denotes the standard basis
in $\mathbb{R}^n$; $e_{ij}$ denotes
the $n \times n$ matrix for which the entry with the coordinates $(i, j)$ is equal to one and all other entries are zero; $\langle \cdot, \cdot \rangle$ is a standard dot product on $\mathbb{R}^n$.

\section{The number of spanning trees of a Ferrers graph}\label{sec:sharp}

\par The original proof of the sharpness of the Ehrenborg's conjecture for Ferrers graphs from the paper \cite{Ferrers} uses the methods of the theory of electrical network. 
In this section, we provide an alternative proof of the fact that does not use this theory. Let us formulate the main statement in the form of the following theorem.

\begin{theorem}\label{thm:sharp}
For any Ferrers graph, the equality holds
\begin{align}\label{eq:sharp}
 \tau(G) = \frac{D(G)}{mn}   
\end{align}
\end{theorem}
To prove this statement, we will be inductively calculating the determinant of the matrix of the following form:

\begin{equation}\label{eq:MinorBlock}
M=
\begin{tikzpicture}[baseline=(m-5-1.base)]
\matrix [matrix of math nodes,left delimiter=(,right delimiter=),
             column sep=0pt,
             row sep=0pt,
             nodes={
                    minimum size=1.4em,
                    anchor=center
                    }
             ] (m) {
       0     &  0     & 0      & \dots  & 0      & -1     & -1     & \dots  & -1    & -1 \\
      a_2    & 0      & 0      & \dots  & 0      & -1     & \;     & \;     & \;    & \;\\
      0      & a_3    & 0      & \dots  & 0      & -1     & \;     & \;     & \;    & \;\\
      \vdots & \vdots & \vdots & \ddots & \vdots & \vdots    & \;     & \;     & \;    & \;\\
        0    &  0     & 0      & \dots  &a_n     & -1     & \;     & \;     & \;    & \;\\
        \;   &  \;    & \;     & \;     & \;     & 0      & b_2    &\dots   & 0     & 0  \\
     \;      &  \;    & \;     & \;     & \;     & \vdots & \vdots & \ddots & \vdots& \vdots\\
      \;     &  \;    & \;     & \;     & \;     & 0      & 0      &\dots   &b_{m-1}& 0 \\ 
      \;     &  \;    & \;     & \;     & \;     & 0      & 0      &\dots   & 0     & b_m \\};

\draw[dashed] ($0.5*(m-1-5.north east)+0.5*(m-1-6.north west)$) -- ($0.5*(m-9-6.south east)+0.5*(m-9-5.south west)$);
\draw[dashed] ($0.5*(m-5-1.south west)+0.5*(m-6-1.north west)$) -- ($0.5*(m-5-10.south east)+0.5*(m-6-10.north east)$);

\draw[dashed] ($(m-2-6.north east)$) -- ($(m-2-10.north east)$);
\draw[dashed] ($(m-2-6.north east)$) -- ($(m-5-6.south east)$);

\draw (m-8-3.north) node {$C^T$};
\draw (m-4-9.north west) node {$C$};

\end{tikzpicture}
\end{equation}
\\
\par The rows of the matrix \eqref{eq:MinorBlock} are ordered in such a way that $ a_2 \geq \dots \geq a_n $ and, similarly, $ b_2 \geq \dots \geq b_m $. The sum of the elements in each of the rows from the second to the $n$-th, as well as in each of the last $ m-1 $ columns is equal to zero. Matrix $C$ denotes a Ferrers matrix, i.e., minus ones in its entries form a Ferrers diagram. In addition, the first row of matrix $ C $ consists entirely of minus ones.
\begin{lemma}\label{thm:mainlem}
The determinant of the matrix $ M $ described above is equal to $(-1)^n \prod_{i=2}^n a_i \prod_{j=2}^m b_j$.
\end{lemma}
First of all, let us show how Theorem
\ref{thm:sharp} follows from the statement above. As already noted, the cofactor of any element in the Laplacian of a graph is equal to the number of spanning trees. By definition, we can always place rows and columns of the Laplacian of the given graph $G$ in such a way that the resulting matrix takes the following form

\begin{align}\label{eq:blocks}
    \begin{pmatrix}
    A & D  \\
    D^T & B
  \end{pmatrix}  
\end{align}
\\ 
Here $ A $ and $ B $ are diagonal matrices of sizes $n \times n $ and $ m \times m $, respectively, and $ D $ is a Ferrers matrix.
Moreover, we can assume that at least the first two rows of $ D $ consist entirely of $ -1 $. This corresponds to the fact that $G$ has no vertices of degree one. Indeed, the removal of vertices of degree one does not affect neither the number of spanning trees nor the right-hand side of the equality \eqref{eq:sharp}.
\par Now consider the matrix that can be obtained from \eqref{eq:blocks} by eliminating the first column and the $(n+1)$-th row. This matrix will have the form \eqref{eq:MinorBlock}.
Combining together Lemma \ref{thm:mainlem} and Kirchhoff’s Matrix-Tree Theorem we get

$$\tau(G) = (-1)^{n+2} (-1)^n \prod_{i=2}^n a_i \prod_{j=2}^m b_j = \frac{\prod_{i=1}^n a_i \prod_{j=1}^m b_j}{mn},$$
which corresponds to the statement of the Theorem \ref{thm:sharp}, since the maximum degrees $ a_1 $ and $ b_1 $ in the Ferrers graph are exactly equal to $ m $ and $ n $ respectively.
\par Now let us return to the proof of the lemma.

\begin{proof}[Proof of Lemma 1]

The proof is by induction on the size of the matrix. The base case is trivial. Suppose that the statement of the Lemma \ref{thm:mainlem} is fulfilled for any matrix of the form  \eqref{eq:MinorBlock}, the size of which does not exceed $m+n-2$. 
\par Consider the maximum $ k $ such that in all rows of $ M $ from the first to the $ k $-th, the number of minus ones is equal to $ m $. We noted earlier that $k \geq 2$.

\begin{figure}[h]
\begin{center}
\begin{tikzpicture}
\matrix [matrix of math nodes,left delimiter=(,right delimiter=),column sep=0pt,
             row sep=0pt,
             nodes={
                    minimum size=1em,
                    anchor=center
                    }] (m) {
      0   & \dots   & 0  & 0   & \dots   & 0   & -1   & -1  &\dots & -1 & -1   &\dots   & -1  \\
      a_2 & \dots& 0     & 0      & \dots & 0   & -1   & -1   &\dots & -1   &-1    &\dots  &-1  \\
    \vdots&\ddots&\vdots & \vdots &\dots &\vdots   & \vdots   &\vdots&\dots & \vdots   &\vdots&\dots & \vdots  \\
      0   &\dots & a_k   & 0      & \dots & 0   & -1 & -1   &\dots & -1   &-1    &\dots  &-1  \\
      0   & \dots& 0     &a_{k+1} & \dots & 0   & -1 & -1   &\dots   & -1  & 0    & \dots     & 0  \\
   \vdots &\dots&\vdots  & \vdots&\ddots & \vdots   & \vdots   & \vdots   &\dots   &\vdots   & \vdots &\dots&\vdots\\
      0   & \dots& 0     & 0      & \dots     & a_n & -1   & *   &\dots  & *   & 0    &\dots     & 0  \\
      
      -1  & \dots   & -1          & -1     & \dots   & *  &0  & b_2   & \dots   & 0   & 0        & \dots      & 0  \\
      \vdots  & \dots    & \vdots    & \vdots      & \dots     & \vdots   & \vdots   & \vdots   &\ddots   & \vdots   & \vdots &\dots&\vdots  \\
      -1   & \dots    & -1     & -1      & \dots     & *  & 0   & 0   & \dots   & b_{m-l}   & 0        & \dots     & 0  \\
      -1  & \dots&-1     & 0& \dots     &0  & 0    & 0   &\dots & 0   &\;b_{m-l+1} & \dots  & 0  \\
    \vdots&\dots&\vdots &\vdots& \dots& \vdots    & \vdots   & \vdots  & \dots  & \vdots   &\vdots    &\ddots  &\vdots  \\
      -1  & \dots&-1     & 0 &\dots    & 0   & 0   & 0   &\dots & 0   & 0        &\dots   & b_m \\};
\begin{scope}[on background layer]
            \node[fit=($0.8*(m-2-1.north west)+0.2*(m-2-2.north west)$)(
            $0.9*(m-4-13.south west)+0.1*(m-3-13.south west)$), draw=gray!30, fill=gray!30, rounded corners] {};
            \node[fit=(m-1-1)(m-13-3), draw=gray!30, fill=gray!30, rounded corners] {};
            \node[fit=($0.8*(m-11-1.north west)+0.2*(m-11-2.north west)$)(m-13-13), draw=gray!30, fill=gray!30, rounded corners] {};
            \node[fit=(m-1-11)(m-13-13), draw=gray!30, fill=gray!30, rounded corners] {};
            \node[fit=($0.8*(m-2-1.north west)+0.2*(m-2-2.north west)$)(m-4-3), fill=gray, rounded corners] {};
            
            \node[fit=($0.8*(m-11-1.north west)+0.2*(m-11-2.north west)$)(m-13-3), fill=gray, rounded corners] {};
            \node[fit=(m-2-11)(m-4-13), fill=gray, rounded corners] {};
            \node[fit=($0.8*(m-11-11.north west)+0.2*(m-11-12.north west)$)(m-13-13), fill=gray, rounded corners] {};
        \end{scope} 

\draw[dashed] ($0.5*(m-1-6.north east)+0.5*(m-1-7.north west)$) -- ($0.5*(m-13-7.south east)+0.5*(m-13-6.south west)$);

\draw[dashed] ($0.5*(m-7-1.south west)+0.5*(m-8-1.north west)$) -- ($0.5*(m-7-13.south east)+0.5*(m-8-13.north east)$);

\node[above=0pt of m-1-1] (top-1) {};
\node[above=1pt of m-1-5] (top-5) {};
\node[above=0pt of m-1-11] (top-11) {};
\node[above=0pt of m-1-13] (top-13) {};

\node[right=4pt of m-1-13] (right-2) {};
\node[right=4pt of m-4-13] (right-4) {};
\node[right=4pt of m-6-13] (right-6) {};
\node[right=4pt of m-9-13] (right-8) {};

\node[rectangle,above delimiter=\{] (del-top-1) at ($0.5*(top-11.south) +0.5*(top-13.south)$) {\tikz{\path (m-1-11.south west) rectangle (m-1-13.south east);}};

\node[above=10pt] at (del-top-1.north) {$l$};

\node[rectangle,right delimiter=\}] (del-right-1) at ($0.5*(right-2.west) +0.5*(right-4.west)$) {\tikz{\path (right-2.north east) rectangle (right-4.south west);}};
\node[right=26pt] at (del-right-1.west) {$k$};

\end{tikzpicture}

\end{center}
\caption{The matrix $M$ with highlighted blocks.}
\label{fig:ColorBlocks}
\end{figure}

\par Subtract the first row of the matrix from all rows from the second to the $ k $-th. After that all these rows will contain a single non-zero number $ a_2, a_3, \dots a_k $, respectively. Hence the determinant of the required matrix $ M $ can be calculated as 
$$
\det(M) = \det(M_1) (-1)^{k-1} \prod_{i=2}^ k a_i,
$$ 
where $ M_1 $ is a square matrix of size $ (n + m) -k $, which is obtained from matrix $ M $ by removing rows from the second to the $ k $-th and columns from the first to the $ (k-1) $-th.

\par Further, we denote by $ l <m $ the number of
minus ones in the $ (k + 1) $-th row of matrix $ M $. Then in each of the last $ l $ rows of matrix $ M_1 $ there is only one nonzero element $ b_ {m-l + 1}, \dots, b_m $ respectively. By eliminating all these rows and the corresponding columns we obtain a new square matrix $ M_2 $ of the form \eqref{eq:MinorBlock} and size $ m + n - k - l $ (which is strictly less than $ m + n-1 $, since $ k \geq 2 $ by definition of $ M $) for which the following relation holds
$$
\text{det}(M_1) = \text{det}(M_2) \prod_{j=m-l+1}^m b_j.
$$
\par Described transition from $ M $ to $ M_2 $ is illustrated in Figure \ref{fig:ColorBlocks} (non-highlighted areas correspond to $ M_2 $ matrix). Note that the property that the first two rows of the matrix have the same amount of minus ones is preserved for matrix $M_2$, since we left the first row of the original matrix $ M $ unchanged.

\par Hence we can apply the induction hypothesis to the matrix $ M_2 $. After that the determinant of the original matrix can be computed as
\[
\text{det}(M) =  \text{det}(M_2) (-1)^{k-1}  \prod \limits_{i=2}^k a_i  \prod  \limits_{j=m-l+1}^m b_j = (-1)^{k-1} (-1)^{n-k+1}\prod \limits_{i=2}^n a_i \prod_{j=2}^m b_j.
\]
\par The obtained equality exactly corresponds to the statement of the lemma.
\end{proof}

\section{Bounds for the graphs with regularity properties}\label{sec:BoundsRegular}
\par In this chapter, we will show that the Ehrenborg's conjecture \eqref{eq:Econj} holds for a one-side regular graph.
The result is most essential in a situation when the number of edges is small relative to the maximum possible. Because currently known estimates in terms of the number of edges, for example \eqref{eq:Bozkurt}, do not give a sufficiently good approximation of the conjecture statement in that case.
\par Section~\ref{sec:generalization} of this chapter is devoted to estimation of the number of spanning trees for a biregular graph with a constrain on the degree of a particular vertex. This question is a natural generalization of the main problem, since the total number of spanning trees is equal to the sum of the coefficients of the spanning trees generating polynomial, and in the section we aim to estimate the partial sums of the coefficients of this polynomial.

\subsection{Ehrenborg's conjecture for one-side regular graphs}
Let us state the main result.

\begin{theorem}\label{thm:oneside_conj}
For a bipartite one-side regular graph
$ G $ the following inequality holds:
$$\tau(G)\leqslant \frac{D(G)}{mn}.$$
\end{theorem}

Further we will need the following technical lemma. The proof of it in a more general setting can be found in \cite{Grone85}. However for the sake of completeness we present a proof of the lemma here.
\begin{lemma}\label{thm:techlem}
Let $A$ be a non-negative semidefinite singular Hermitian matrix of size $n \times n$.
Denote by $ (A | (i, i)) $ the matrix of order $ n-1 $ obtained from $ A $ by deleting the $ i $-th row and the $ i $-th column, then
\begin{equation}\label{eq:techlem}
\sum_{i=1}^n a_{ii}\det(A|(i,i)) \leq \left(\frac{n}{n-1}\right)^{n-1} \prod_{i=1}^n a_{ii}.
\end{equation}
\end{lemma}
\begin{proof}
If one of the diagonal elements, say $a_{kk}$, of matrix $A$ is equal to 0, then
$0=\langle Ae_k,e_k\rangle=
\langle \sqrt{A}e_k,\sqrt{A}e_k\rangle$,
hence $\sqrt{A}e_k=0$ and $Ae_k=\sqrt{A}\cdot \sqrt{A} e_k=0$. Thus for matrix $A$ all its entries of $k$-th row and $k$-th column are equal to zero. Hence all the terms
on the left of \eqref{eq:techlem} are zero and \eqref{eq:techlem} holds trivially.\\
Now assume that $a_{kk}>0$ for all
$k=1,\ldots,n$. 
Denote $B = DAD$, where $ D $ is a diagonal matrix for which $d_{ii} = 1/\sqrt{a_{ii}}$. 
Then $b_{ii} = 1$ for any $i$, and the original inequality is equivalent to the following:
$$
\sum_{i=1}^n \det(B|(i,i)) \leq \left(\frac{n}{n-1}\right)^{n-1}.
$$
We denote the eigenvalues of $ B $ as
$\beta_1 \geq \beta_2 \geq \dots \geq \beta_n = 0$.
Then
$$
\prod_{i=1}^n(\beta_i-t)
=\det(B-tI)=\det(B)-t\sum_{i=1}^n \det(B|(i,i))+
\ldots
$$
hence $$\sum_{i=1}^n \det(B|(i,i))=
\beta_1\beta_2 \dots \beta_{n-1}\leq
\left( \frac{\beta_1 + \beta_2 + \dots + \beta_{n-1}}{n-1}\right)^{n-1}= \left( \frac{n}{n-1}\right)^{n-1}.
$$
\end{proof}

Now with help of Lemma \ref{thm:techlem} we are able to obtain an upper bound for the minor of order $ n-1 $ via the product of diagonal elements.

\begin{proof}[Proof of the Theorem  \ref{thm:oneside_conj}]
First we enumerate vertices of $ G $ so that the diagonal elements of the Laplacian $ L(G) $ were
presented in the following order: $ a_1, a_2, \dots, a_n, b_1, b_2, \dots, b_m $. For square matrix $ X $ of size $ m + n $ we define the following transformation.
$$
S_{i, j}(X) = T_{n+j, i}\left(\frac{1}{b_j}\right) X  T_{i, n+j}\left(\frac{1}{b_j}\right),
$$
where $T_{i,j}(\lambda):=I+\lambda e_{ji}$ 
is a transvection.
 \par We apply transformation $ S_{i, j} $ to $ L(G) $ for every pair of coordinates  $ i \in [n] $, $ j \in [m] $, for which $ L(G)_{i, n + j} $ is equal to minus one.
 
The resulting matrix  $\widehat{L(G)} $ will remain non-negative semidefinite, symmetric, singular, and can be presented in the following form:
\begin{center}
\begin{equation*}
\widehat{L(G)} =
\begin{tikzpicture}[baseline=(m-4-1.base)]

\matrix [matrix of math nodes,left delimiter=(,right delimiter=),row sep=0.1cm,column sep=0.1cm] (m) {
      \; & \; & \;   & 0 & 0 & \dots & 0 \\
       \; & C &\;& \vdots & \vdots & \dots & \vdots\\
        \;    &   \;  &\;& 0 & 0 & \dots & 0 \\
         0    &  \dots   & 0  & b_1 &  0 & \dots & 0 \\
         0    &  \dots   & 0  & 0 & b_2 &\dots & 0 \\
       \vdots & \dots & \vdots & \vdots & \vdots & \ddots & \vdots\\
       0    &   \dots   & 0   & 0 & 0 &\dots & b_m\\ };

\draw[dashed] ($0.5*(m-1-3.north east)+0.5*(m-1-4.north west)$) -- ($0.5*(m-7-4.south east)+0.5*(m-7-3.south west)$);

\draw[dashed] ($0.5*(m-3-1.south west)+0.5*(m-4-1.north west)$) -- ($0.5*(m-3-7.south east)+0.5*(m-4-7.north east)$);

\end{tikzpicture}
\end{equation*}
\end{center}
 \par The transformations set to zero the upper right and the lower left blocks. Matrix $ C $ of size $ n \times n $ can be fully described in a following way:

\[
    c_{ij}= 
\begin{cases}
    a_i - \sum\limits_{k:(i, k)\in E} \frac{1}{b_k},& \text{if } i=j\\
     -\sum\limits_{\substack{k:(i, k) \in E \\ (j, k) \in E}} \frac{1}{b_k},& \text{otherwise}
\end{cases}
\]
Similarly to $\widehat{L(G)}$, matrix $ C $ is non-negative semidefinite, singular and symmetric, which means that Lemma \ref{thm:techlem} is applicable for $C$: 
\begin{equation}\label{eq:lem_conseq}
\sum_{i=1}^n c_{ii} \det(C|(i,i)) \leq \left(\frac{n}{n-1}\right)^{n-1} \prod_{i=1}^n c_{ii}.
\end{equation}
In the matrix transformations applied above, we subtracted the rows and columns of the second component from the rows and columns of the first
component, therefore all cofactors 
$ \det (C | (i, i)) $ are still equal to each other, hence \eqref{eq:lem_conseq} transforms into
\begin{equation}\label{eq:result}
\tau(G)  \leq \left(\frac{n}{n-1}\right)^{n-1}  \frac{\prod_{i=1}^n (1 - \sum\nolimits_{(i, k) \in E} \frac{1}{a_ib_k})}{\sum_{i=1}^n (a_i - \sum\nolimits_{(i, k) \in E} \frac{1}{b_k})}\prod_{i=1}^n a_{i} \prod_{j=1}^m b_{j}.
\end{equation}

Until this moment we did not use the regularity condition, but now it is time to do it. We rewrite the number of edges as $ | E |  =  \varepsilon mn $, where $ \varepsilon \in \left[0, 1 \right] $, then, without loss of generality, we will assume that the second component is regular, that is $ b_j = \varepsilon n $ for any $ j \in [m] $, hence the inequality \eqref{eq:result} can be simplified: 
\begin{equation*}
\begin{gathered}
\tau(G)  \leq \left(\frac{n}{n-1}\right)^{n-1}  \frac{(1 - \frac{1}{\varepsilon n})^n}{m(\varepsilon n - 1)}\prod_{i=1}^n a_{i} \prod_{j=1}^m b_{j} =  \\
= \frac{(\varepsilon n - 1)^{n-1}}{\varepsilon^n(n - 1)^{n-1}} \frac{D(G)}{mn}.
\end{gathered}
\end{equation*}

It remains to show that the coefficient obtained beside $ \frac{D (G)}{mn} $ does not exceed one. To do this, first check that the expression $\frac{(\varepsilon n - 1)^{n-1}}{\varepsilon^n(n - 1)^{n-1}}$ is increasing in $\varepsilon$:

\begin{equation*}
\begin{gathered}
\left(\frac{(\varepsilon n - 1)^{n-1}}{\varepsilon^n}\right)' = \left(\frac{(\varepsilon n - 1)^{n-2}}{\varepsilon^{n+1}}\right) (\varepsilon(n-1)n - (\varepsilon n - 1)n) =  \\
= \left(\frac{(\varepsilon n - 1)^{n-2}}{\varepsilon^{n+1}}\right) (1 - \varepsilon )n \geq 0.
\end{gathered}
\end{equation*}

The last inequality holds for $ 1/n \leq \varepsilon \leq 1 $, which corresponds to the existing constraints, since $ 1 \leq b_j = \varepsilon n \leq n $. Now it will be enough to consider the extreme case $ \varepsilon = 1 $. In this case, the equality holds. 

\end{proof}

\subsection {Generalization for a biregular graph}\label{sec:generalization}
\par In this chapter, using the 
technique similar to that we applied in the proof of Theorem \ref{thm:oneside_conj}, we give an upper bound for the number of spanning trees with constraints on the degree of a particular vertex. This result can be viewed as a generalization of the Theorem \ref{thm:oneside_conj} for biregular
bipartite graphs.
\par Further we will consider a bipartite biregular graph $ G = (V, E) $, in which the sizes of the first and second components are equal to $ n $ and $ m $, respectively. We denote by $ a $ and $b = an/m$ the degrees of the vertices from the first and second components, respectively.\\
Let $ v $ be an arbitrary vertex from the second component. Let $ \tau^v _{> k}(G) $ denote the number of spanning trees of the graph $ G $ in which  $ v $ has degree strictly greater than $ k $. Similarly, $ \tau^v_{\leq k} (G) $ will denote the number of spanning trees of the graph $ G $ in which  $ v $ has degree at most $ k $.
\begin{theorem}\label{thm:twostats}
Assume that $k=\theta\cdot b/a$ is a non-negative integer number which is

less than $b$.
Then, under the above assumptions, the following inequalities hold.
\begin{enumerate}[label=(\roman*)]
    \item $\tau^v_{>k}(G) \leq \left(\theta^{-1}\exp\left(\frac{\theta-1}\theta\right)\right)^k  \frac{ a^{n} b^{m}}{mn}$ if $\theta \geq 1$
    \item $\tau^v_{\leq(k+1)}(G) \leq \left(\theta^{-1}\exp\left(\frac{\theta-1}\theta\right)\right)^k  \frac{ a^{n} b^{m}}{(m-1)n}$ if $\theta \leq 1$
\end{enumerate}

\end{theorem}
In order to prove Theorem \ref{thm:twostats} we will use the same technique that we applied while estimating the determinant in Theorem \ref{thm:oneside_conj}, however this time we will need the following generalization of Kirchhoff's Matrix-Tree Theorem (see, e.\,g., \cite[p. 35, problem 10]{L}).
\begin{theorem*}(Generalized Kirchhoff's Theorem)
For an arbitrary graph $ G = (V, E) $, a generalized Laplacian is a matrix $ L(G) $, depending on free variables $ x_1, \dots x_{| V |} $ defined as follows
 \[
    a_{ij}= 
\begin{cases}
    \sum_{(i, k) \in E}x_ix_k,  & \text{если } i=j\\
     -x_ix_j ,& \text{если } v _iv_j \in E\\
     0, & \text{иначе}
\end{cases}
\]
\par For such matrix the cofactors of all elements are equal as polynomials to the generating polynomial $ P_G $ for the number of spanning trees, that is, for any pair of indexes $(i, j)$ 
\begin{equation}\label{eq:glaplacian}
   \det(L(G)|(i,j)) =  \sum_{t \in T_G} \prod_{i=1}^{|V|}x_i^{\deg_t(i)} =: P_G(x_1, \dots, x_{|V|}). 
\end{equation}

\par Here $ T_G $ denotes the set of all spanning trees of the graph $ G $ and $ \deg_t(i) $ is the degree of the $ i $-th vertex in the tree $t$.
\end{theorem*}
It is easy to see that the substitution $ x_1 = \dots = x_{| V |} = 1 $ gives the original Kirchhoff's theorem.
In general, \eqref{eq:glaplacian} calculates the sum of the weights of all spanning trees
in a graph where the weight of an edge between $ v_i$ and $v_j $ is equal to $ x_ix_j $.
The generalized Laplacian corresponds to the
quadratic form $\sum_{v_iv_j\in E} x_ix_j(f(v_i)-f(v_j))^2$
on the space $ \mathbb{R}^V $, hence
for non-negative values of the variables
$ x_i $
its non-negative definiteness
immediately follows.

\begin{proof} [Proof of the Theorem \ref{thm:twostats}]

Suppose that the vertices of the first component correspond to the variables $x_1$, $x_2$, $\dots$, $x_n$, and the vertices of the second component correspond to $y_1$, $y_2$, $\dots$, $y_n$. Without loss of generality, we will assume that vertex $ v $ corresponds to variable $ y_1 $. Now we substitute
the value 1 for each variable $ x_1, x_2, \dots, x_n, y_2, \dots, y_m $ in the generating polynomial. What remains is a polynomial in a single variable $ y_1 $.

\par After the substitution described above matrix $ L(G, y_1) $ in a suitable basis will have the following form:
\newpage
\begin{equation*}
L(G, y_1)=
\begin{tikzpicture}[baseline=(m-5-1.base)]
\matrix [matrix of math nodes,left delimiter=(,right delimiter=),
             column sep=0pt,
             row sep=0pt,
             nodes={
                    minimum size=1.4em,
                    anchor=center
                    }
             ] (m) {
        {\scriptstyle (a-1) + y_1 } & \dots  & 0   & \;  & \;  & \;   &-y_1  & \;  & \;  & \;  \\
       \vdots   & \ddots  & \vdots   & \;  & 0  & \;   &\vdots  & \;  & *  & \;  \\
       0  & \dots  & {\scriptstyle (a-1) + y_1 }   & \;  & \;  & \;   &-y_1  & \;  & \;  & \;  \\
       \;  & \;  & \;   & a  & \dots  & 0              &0   & \;  & \;  & \; \\
       \;  & 0  & \;    & \vdots  & \ddots  & \vdots   &\vdots  & \;  & *  & \;  \\
       \;  & \;  & \;   & 0  & \dots  & a              &0  & \;  & \;  & \;  \\
       -y_1  & \dots & -y_1   & 0  & \dots  & 0           & by_1   & 0   & \dots  & 0\\
       \;  & \;  & \;         & \;  & \;  & \;             &0   & b   & \dots  & 0 \\
       \;  & *  & \;          & \;  & *  & \;              &\vdots  & \vdots   & \ddots  & \vdots \\
       \;  & \;  & \;         & \;  & \;  & \;             &0  & 0   & \dots  & b \\};

\draw[dashed] ($0.4*(m-1-3.north west)+0.6*(m-1-4.north east)$) -- ($0.4*(m-1-3.north west)+0.6*(m-1-4.north east) + (0, -5.7)$);

\draw ($0.6*(m-1-6.north west)+0.4*(m-1-7.north east)$) -- ($0.6*(m-1-6.north west)+0.4*(m-1-7.north east) + (0, -5.7)$);

\draw[dashed] ($0.4*(m-1-7.north west)+0.6*(m-1-8.north east)$) -- ($0.4*(m-1-7.north west)+0.6*(m-1-8.north east) + (0, -5.7)$);

\draw[dashed] ($0.5*(m-3-1.south west)+0.5*(m-4-1.north west)$) -- ($0.5*(m-3-1.south west)+0.5*(m-4-1.north west) + (7.3, 0)$);

\draw ($0.5*(m-6-1.south west)+0.5*(m-7-1.north west)$) -- ($0.5*(m-6-1.south west)+0.5*(m-7-1.north west) + (7.3, 0)$);

\draw[dashed] ($0.5*(m-7-1.south west)+0.5*(m-8-1.north west)$) -- ($0.5*(m-7-1.south west)+0.5*(m-8-1.north west) + (7.3, 0)$);

\node[right=4pt of m-2-10] (right-2) {};
\node[right=4pt of m-3-10] (right-4) {};
\node[right=4pt of m-6-10] (right-6) {};
\node[right=4pt of m-9-10] (right-8) {};

\node[rectangle,right delimiter=\}] (del-right-1) at ($0.95*(right-2.west) +0.05*(right-4.west)$) 
{\tikz{\path (right-2.south east) rectangle (right-4.north west);}};
\node[right=26pt] at (del-right-1.west) {$b$};

\end{tikzpicture}
\end{equation*}
\par Here, $ * $ denotes minus one or zero according to the definition, and the sum of the elements in each row and in each column is equal to zero.

\par \par Next, as in Theorem \ref{thm:oneside_conj} we apply transformations $S_{i,j}$ to the generalized Laplacian, and then use Lemma \ref{thm:techlem}. In order to apply the lemma we need to calculate only diagonal elements of
the resulting matrix. Let us write them down:
\[
    a_{ii}= 
\begin{cases}
    (1-\frac{1}{b})(y_1 + (a-1)),& \text{при } 1 \leq i \leq b\\
     (1-\frac{1}{b})a,& \text{при } b < i \leq n\\
     by_1,& \text{при } i = n+1\\
     b,& \text{при } n+1 < i \leq m
\end{cases}
\]
\par We apply Lemma \ref{thm:techlem} to the resulting matrix and obtain the following inequality for any positive $ y_1 $:

$$
P_G(y_1)  \leq \left(\frac{n}{n-1}\right)^{n-1} \frac{(y_1 + (a-1))^b a^{n-b} (1 - \frac{1}{b})^n}{(1 - \frac{1}{b})(by_1 + b(a-1) + a(n-b))} b^m y_1.
$$
\par As in the previous proof, we denote
$b = \varepsilon n$, then the inequality above can be simplified using two relations: $an = bm$ and $\frac{(\varepsilon n - 1)^{n-1}}{\varepsilon^n(n - 1)^{n-1}} < 1$
for $\varepsilon=b/n$:
$$
P_G(y_1)  \leq  \frac{(y_1 + (a-1))^{b} a^{n-b} b^{m} y_1}{(y_1 + (m-1))n} := Q(y_1).
$$
\par We can represent the generating polynomial in the following way: $P_G(y_1) = c_by_1^b + \dots + c_2y_1^2 + c_1y_1$ (note that the free term of $ P_G $ is always zero). In these terms $\tau^v_{>k}(G) = \sum_{i = k+1}^b c_i$ and $\tau^v_{\leq(k+1)}(G)~=~\sum_{i = 1}^{k+1}c_i$.
Hence using the fact that all coefficients of $ P_G $ are non-negative, we obtain that for any $ y_1 \geq 1 $ the following inequalities hold:
\begin{align}\label{eq:fstcase}
\begin{gathered}
    \tau^v_{>k}(G) = c_{k+1} + \dots + c_b \leq \\
 \leq   c_{k+1} + c_{k+2}y_1 + \dots + c_by_1^{b-k-1}
\leq \frac{P_G(y_1)}{y_1^{k+1}} \leq  \frac{Q(y_1)}{y_1^{k+1}}.
\end{gathered}
\end{align}
\par Similarly, for $0 < y_1 \leq 1$ we have
\begin{align}\label{eq:sccase}
\tau^v_{\leq k+1}(G) = c_{1} + \dots + c_{k+1} \leq \frac{P_G(y_1)}{y_1^{k+1}} \leq  \frac{Q(y_1)}{y_1^{k+1}}.
\end{align}
\par Now we perform a substitution
$y_1= 1 + \frac{a(\theta - 1)}{(a - \theta)}$ in polynomial $\frac{Q(y_1)}{y_1^{k+1}}$, where $k = \theta b/a$.
Note that we can assume that $ \theta <a $, since the case $k \geq b$ is not interesting in terms of Theorem \ref{thm:twostats}.
\par For the first statement of the theorem we have $ \theta \geq 1 $ which implies that the chosen $ y_1 $ is at least one and the inequality \eqref{eq:fstcase} holds. Similarly, for the second statement \eqref{eq:sccase} holds. Hence in order to prove both statements we need to bound the expression obtained after the substitution:
\begin{align}\label{eq:both}
     \frac{(a-1)^{b-k}}{(a - \theta)^{b-k} \theta^{k} (\frac{a(\theta - 1)}{m(a - \theta)} + 1)} \frac{ a^{n} b^{m}}{mn}.
\end{align}
\par For the first statement we have $\theta \geq 1$, hence $\frac{a(\theta - 1)}{m(a - \theta)} > 0$ and it is possible to bound \eqref{eq:both} from above, which gives us the following inequality
$$
\tau^v_{>k}(G) \leq \theta^{-k} \left(1 + \frac{(\theta - 1)}{(a- \theta)}\right)^{(a-\theta) k/\theta} \frac{ a^{n} b^{m}}{mn} \leq \left(\theta^{-1}\exp\left(\frac{\theta-1}\theta\right)\right)^k  \frac{ a^{n} b^{m}}{mn},
$$
\par which corresponds to the statement of the theorem. Moreover, it is easy to verify that the expression under the power of $ k $ does not exceed one for any $ \theta > 0 $. Thus, we have improved the general estimate of the Theorem~\ref{thm:oneside_conj}.
\par Now let us prove the second statement.
Since $ a \geq1 $, for any $ \theta \leq 1 $ we have
\begin{align*}
    \frac{a(1 - \theta)}{m(a - \theta)} < \frac{1}{m}.
\end{align*}
\par Using the above inequality and \eqref{eq:both} we get
$$
\tau^v_{\leq k+1}(G) \leq  \left(\theta^{-1}\exp\left(\frac{\theta-1}\theta\right)\right)^k  \frac{ a^{n} b^{m}}{(m-1)n}.
$$
\par This proves the theorem. Note that the bound we got for the second statement
differs only by the factor $\frac{1}{(1-1/m)}$, which tends to one as the size of the corresponding component of the graph increases.
\end{proof}

\begin{remark}
Consider the case $ \theta = k = 0 $. Then the inequality (ii) of the Theorem \ref{thm:twostats} turns out to be weaker than the bound given by Theorem \ref{thm:oneside_conj} for the number of spanning trees with the leaf vertex
$v$.
However, in that situation we can use inequality \eqref{eq:both} and get an upper bound $(1-1/a)^b \frac{a^nb^m}{(m-1)n}$ which is better than what Theorem \ref{thm:oneside_conj} gives and it is also sharp for the complete bipartite graphs.
\end{remark}

\section*{Acknowledgment}
The author is grateful to her advisor Fedor Petrov for his helpful comments and suggestions concerning this paper. The main results of this paper were formulated in the author’s bachelor thesis.


\begin{thebibliography}{10}


\bibitem{Kirh}
G. Kirchhoff, \emph{Ueber die Auflösung der Gleichungen, auf welche man bei der Untersuchung der linearen Vertheilung galvanischer Ströme geführt wird}, Ann. Phys. 148 (1847), pp. 497--508. 

\bibitem{CDS80}
D. Cvetkovi\'c, M. Doob, H. Sachs. 
 \emph{Spectra of Graphs—Theory and Application}, Academic Press, New York (1980). 


\bibitem{Cheng}
C. Cheng, \emph{Maximizing the total number of spanning trees in a graph: Two related problems in graph theory and optimization design theory}, J. Combinat. Theory B 31 (1981), pp. 240--248.

\bibitem{Grone85}
B. Grone, \emph{An inequality for the second immanant}, Linear and Multilinear Algebra 18:2 (1985), pp.
147--152.

\bibitem{GM}
R. Grone, R. Merris, \emph{A bound for the complexity of a simple graph}, Discrete Math. 69 (1988), pp.
97--99.


\bibitem{L} L. Lovász, \emph{Combinatorial 
problems and exercises}, North-Holland Pub. Co, Amsterdam--London (1993)


\bibitem{Ferrers}
R. Ehrenborg, S. Willigenburg, \emph{Enumerative properties of Ferrers graphs},
Disc. Comp.Geom. 32 (2004), pp. 481--492.

\bibitem{Bozkurt}
\c{S}. B. Bozkurt, \emph{Upper bounds for the number of spanning trees of graphs}, J. Inequal. Appl. 2012 (2012), article no. 269.

\bibitem{KS19}
S. Klee, M. T. Stamps, \emph{Linear algebraic techniques for spanning tree enumeration}, arXiv:1903.04973 (2019)


\end{thebibliography}
\end{document}